\documentclass[12pt]{elsarticle}
\usepackage{verbatim}
\usepackage{amssymb,amsmath,amsthm, amsfonts, mathrsfs, bbm,dsfont}
\usepackage[utf8]{inputenc}
\usepackage{mathrsfs}
\usepackage{amsfonts}
\newtheorem{theorem}{Theorem}[section]
\newtheorem{lemma}{Lemma}[section]

\newtheorem{corollary}{Corollary}[section]

\newtheorem{definition}{Definition}[section]
\newtheorem{example}{Example}[section]
\newtheorem{remark}{Remark}

\usepackage{extarrows}
\usepackage[top=1in, bottom=1in, left=1.25in, right=1.25in]{geometry}

\usepackage{setspace}
\begin{document}
\begin{spacing}{1.1}
\begin{frontmatter}

\title{Two tuples of noncommutative Orlicz sequence spaces and some geometry properties\tnoteref{mytitlenote}}
\tnotetext[mytitlenote]{The research has been supported by Nature Science Foundation of Hebei Province (No. A2023404001)and Three Three Three Talent
Project funded project of Hebei Province 2023 (No. C20231019).}


\author[mymainaddress]{Ma Zhenhua  \corref{Ma Zhenhua}}
\cortext[Ma Zhenhua]{Corresponding author}
\ead{mazhenghua\_1981@163.com}

\author[mysecondaddress]{Jiang Lining}
\cortext[Jiang Lining]{}
\ead{jainglining@bit.edu.cn}

\address[mymainaddress]{Hebei University of Architecture, Zhangjiakou, 075024, P. R. China}
\address[mysecondaddress]{Beijing Institute of Technology, Beijing, 100081, P. R. China}

\begin{abstract}
The primary contribution of this study lies in proposing  a new concept termed $2$-tuples of noncommutative Orlicz sequence spaces $\bigoplus\limits_{j=1}^{2}S_{\varphi_{j},p}$, where $S_{\varphi_{j}}$ denotes a noncommutative Orlicz sequence space. By leveraging the three-line theorem, we establish the Riesz-Thorin interpolation theorem for $\bigoplus\limits_{j=1}^{2}S_{\varphi_{j},p}$. As applications, we derive bound for the nonsquare and von Neumann-Jordan constant of noncommutative Orlicz space $S_{\varphi_{s}} (0<s\leq1)$, where $\varphi_{s}$ is an intermediate function.
\end{abstract}

\begin{keyword}
compact operator \sep Orlicz function\sep Riesz-Thorin interpolation\sep von Neumann Jordan constant
\MSC[2010] 46L52\sep 47L10\sep46A80
\end{keyword}

\end{frontmatter}

\section{Preliminaries}
Interpolation methods constitute fundamental tools in modern mathematical analysis. Clarkson introduced a pivotal inequality—now known as Clarkson's inequality to study the uniform convexity of $L_{p}$ spaces. Building on this foundation, the noncommutative Riesz-Thorin interpolation theorem was later employed to derive Clarkson-type inequalities for $L_{p}$ spaces\cite{Clarkson}. In \cite{Xu},  the author used the noncommutative Riesz-Thorin interpolation theorem to get the Clarkson inequality of noncommutative $L_{p}$ spaces.

The principal objective of this paper is to investigate Riesz-Thorin interpolation theorem on the $2$-tuples noncommutative Orlicz sequence spaces $\bigoplus\limits_{j=1}^{2}S_{\varphi_{j},p}$, which yields the Clarkson inequality of noncommutative Orlicz sequence spaces $S_{\varphi}$. As an application, the bound of von Neumann Jordan constant of noncommutative Orlicz sequence spaces $S_{\varphi_{s}} (0<s\leq1)$ is given.

Let $\mathcal{B}(\mathcal{H})$ denote the algebra of all bounded linear operators in a complex, separable Hilbert space $\mathcal{H}$. If $T$ in $\mathcal{B}(\mathcal{H})$ is a compact operator and $T^{*}$ is its adjoint
$T^{*}T: \mathcal{H}\rightarrow\mathcal{H}$ is positive and compact, and so is $T^{*}T$ has a unique positive square root denoted by $|T| = (T^{*}T)^{1/2}$.  Correspondingly,
$\mathcal{H}$ has an orthonormal basis $\{e_{k}\}_{k}$, each of which is an eigenvector of $|T|$.

The eigenvalue $\lambda_{k}(|T|)$  corresponding to  each  eigenvector $e_{k}$ is necessarily non-negative, and it is  called a singular value of  the operator $T$, which is denoted by $s_{k}(T)$.
Without loss of generality, by reordering the members of the orthonormal basis $\{e_{k}\}$ if needed, we assume $s_{k}(T)$ are arranged in decreasing order, that is, $s_{k+1}(T) \leq s_{k}(T)$ for $k=1,2,\cdots.$

If $\sum\limits^{\infty}_{k=1}s_{k}(T)<\infty,$ the  operator $T:\mathcal{H}\rightarrow \mathcal{H}$ is called a trace class operator. The set of all trace class operators is denoted by $S_{1}(\mathcal{H})$. It can be shown that $S_{1}(\mathcal{H})$ is a Banach space in which the trace norm $\|\cdot\|_1$ is given by $\|T\|_{1}=\sum^{\infty}\limits_{k=1}s_{k}(T)$. For more details of $S_{1}(\mathcal{H})$ please refer to \cite{GK,Murphy}.

The  Schatten classes $S_{p}(\mathcal{H})$ is given as follows:
for any $0<p<\infty$, let $T$ be a compact operator of $\mathcal{H}$, and we define
$$S_{p}(\mathcal{H})=\left\{T: \ {\rm tr}\left(|T|^{p}\right)=\sum\limits_{n\geq1}s_{n}\left(T^{p}\right)=\sum\limits_{n\geq1}\left(s_{n}\left(T\right)\right)^{p}<\infty\right\}.$$
For any $T\in S_{p}(\mathcal{H})$, the norm is
$$\|T\|_{p}=\left(\sum\limits_{n\geq1}\left(s_{n}\left(T\right)\right)^{p}\right)^{\frac{1}{p}}.$$
Especially, $S_{2}(\mathcal{H})$ is called Hilbert-Schmidt operator space \cite{Bridges}, and the norm is $$\|T\|_{2}=\left(\sum\limits_{n\geq1}\left(s_{n}\left(T\right)\right)^{2}\right)^{\frac{1}{2}}.$$ Literatures on the Schatten class operators and their applications are very rich, please see the very interesting recent papers (\cite{Bo,WKG,MI1,MI2,SVD})
and references cited therein. Especially,  Gil',M.I. extended  some
useful results on determinants of Schatten–von Neumann operators to the operators with the Orlicz type norms and modular function, respectively (\cite{MI2,MI4}).

A convex function $\varphi: [0,\infty)\rightarrow[0,\infty]$ is called an Orlicz function
if it is nondecreasing and continuous for $\alpha\geq0$ and satisfying $\varphi(0)=0,\,\varphi(\alpha)>0$ and $\varphi(\alpha)\rightarrow\infty$ as $\alpha\rightarrow\infty$. The Orlicz function $\psi$ completementary to $\varphi$ can be defined by $\psi(y)=\sup\{x|y|-\varphi(x): x\geq0, y\in \mathbf{R}\}$.
Also $\varphi,\psi$ can be given by
$$\varphi(x)=\int_{0}^{|x|}p(t)dt, \psi(y)=\int_{0}^{|y|}q(s)ds.$$
A pair of complementary Orlicz functions satisfy the Young inequality $|xy|\leq \varphi(x)+\psi(y)$  iff $x=q(|y|){\rm sign}\ y$ or $y=p(|x|){\rm sign}\ x$.

Let $\varphi$  be an Orlicz function and  $l^{0}$ be the set of all real sequences. We define the Orlicz sequence space as follows
$$l_{\varphi}=\left\{x\in l^{0}: \rho_{\varphi}(\lambda x)=\sum^{\infty}_{i=1}\varphi\left(\lambda x(i)\right)<\infty,  \exists \lambda>0\right\}.$$
It is well known that $l_{\varphi}$ is equipped with so called the Luxemburg norm and Orlicz norm respectively
$$\|x\|_{(\varphi)}=\inf\left\{\lambda>0: \rho_{\varphi}\left( \frac{x}{\lambda}\right)\leq1\right\},$$
$$\|x\|_{\varphi}=\sup\left\{\sum^{\infty}_{i=1}x(i)y(i): \rho_{\psi}(y)\leq1\right\}.$$
Further we say an Orlicz function $\varphi$ satisfies condition $\delta_{2}$ near zero, shortly $\varphi\in\delta_{2}(0)$, if there exist  constant $u_{0}>0$  and $k>2$ such that $\varphi(2u)\leq k\varphi(u)$ for all $|u|\leq u_{0}$. For the background of Orlicz function and Orlicz space please see \cite{Chen,Rao}.

Now, if $\varphi$ is an Orlicz function and $T$ a compact operator of $\mathcal{H}$, by functional calculus, one has $${\rm tr}(\varphi(T))= \sum\limits^{\infty}_{k=1}\langle \varphi(T)e_{k},e_{k}\rangle=\sum\limits_{n\geq1}\varphi(s_{n}(T))=\sum\limits_{n\geq1}s_{n}(\varphi (T)).$$ Hence, the following definition of the noncommutative Orlicz sequence spaces could be given:
\begin{definition}\cite{Ma}
If $T\in \mathcal{B}(\mathcal{H})$ is a compact operator, we define the noncommutative Orlicz sequence spaces $S_{\varphi}(\mathcal{H})$ as follows:
$$S_{\varphi}(\mathcal{H})=\left\{T: {\rm tr}\left(\varphi(\lambda T)\right)=\sum_{n\geq1}\varphi(\lambda s_{n}(T))<\infty ,  \exists\lambda>0\right\},$$
where  $\rm {tr}$ is the usual trace functional of $\mathcal{B}(\mathcal{H})$.
\end{definition}

For convenience, we use $S_{\varphi}$ to denote $S_{\varphi}(\mathcal{H})$, which could be equipped with the Luxemburg norm as
\begin{eqnarray*}
\|T\|_{(\varphi)}&=&\inf\left\{\lambda>0, {\rm tr}\left(\varphi\left( \frac{T}{\lambda}\right)\right)\leq1\right\}\\
&=& \inf\left\{\lambda>0, \sum_{n\geq1}
\varphi\left(\frac{s_{n}\left(T\right)}{\lambda}\right)\leq1\right\}.
\end{eqnarray*}
 Also, one can define the following norm which is named as Orlicz norm
$$\|T\|_{\varphi}=\sup\left\{{\rm tr}(|TB|):  {\rm tr}(\psi(B))\leq1\right\},$$
where $\psi$ is the complementary function of $\varphi$. It is easy to know $(S_{\varphi}, \|T\|_{(\varphi)})$ and $(S_{\varphi}, \|T\|_{\varphi})$ are Banach spaces. Also one has the following relation of these  norms
$$\|T\|_{(\varphi)}<\|T\|_{\varphi}\leq2\|T\|_{(\varphi)}.$$

\begin{remark}
In the case $\varphi(T)=|T|^{p},\,\,1\leq p<\infty$ for any compact operator $T\in \mathcal{B}(\mathcal{H})$,  $S_{\varphi}$ is nothing but the Schatten
classes $S_{p}$ and the Luxemburg norm is
$$\|T\|_{p}={\rm tr}(|T|^{p})=\sum_{n\geq1}\left(( s_{n}(T))^{p}\right)^{\frac{1}{p}}.$$
Especially, if $p=1$, $S_{\varphi}=S_{1}$ and if $p=2$, $S_{\varphi}$ is the Hilbert-Schmidt operator space $S_{2}$ which satisfies $\varphi\in\delta_{2}(0)$.
\end{remark}
For more information on the theory of noncommutative Orlicz spaces we refer the reader to \cite{Rashed,Rashed1,Ghadir,Kunze,Zhenhua,Muratov}.

\section{Two tuples of Noncommutative Orlicz Sequence Spaces}
In this section, we will present a new concept $2$-tuples of noncommutative Orlicz spaces and  give some norm inequalities. In order to get the Riesz-Thorin interpolation theorem,  equivalent definition of Luxmburg norm must be given. As a corollary, the Clarkson inequality of noncommutive $S_{p}$ space could be get. The main ideas of proof in this paper are derived from literatures \cite{Xu} and \cite{Rao}.

Let $\mathcal{N}=B(\mathcal{H})\oplus B(\mathcal{H})$ be the $2$-th bounded linear operators direct sum of $\mathcal{H}$, then $\mathcal{N}$ acts on the direct sum Hilbert space $\mathcal{H}\oplus\mathcal{H}$ coordinatewise: $$Tx=(T_{1}, T_{2})(x_{1},x_{2})=\sum\limits^{2}_{j=1}T_{j}x_{j},$$ where $T_{j}\in B(\mathcal{H}), x_{j}\in \mathcal{H}$ and $j=1,2.$ Then $\mathcal{N}_{+}=B(\mathcal{H})_{+}\oplus B(\mathcal{H})_{+}$.

We define $\upsilon:\mathcal{N}_{+}\rightarrow \mathds{C}$ by $\upsilon(T)=\sum\limits^{2}_{j=1}{\rm tr}(T_{j})={\rm tr}\left(
                                                                      \begin{array}{cc}
                                                                        T_{1} & 0 \\
                                                                        0 & T_{2} \\
                                                                      \end{array}
                                                                    \right)
,$ where  $\rm{tr}$ is  normal semifinite faithful trace on $B(\mathcal{H})$, then $\upsilon$ is  normal faithful normal trace on $\mathcal{N}$.
\begin{definition}
Let $\varphi=(\varphi_{1},\varphi_{2})$ be $2$-tuples of N function. For each $p\geq1$, consider the following direct sum space: $$\bigoplus\limits_{j=1}^{2}S_{\varphi_{j},p}
=\{T=(T_{1},T_{2}): T_{j}\in S_{\varphi_{j}}(\mathcal{H}), j=1,2\}$$
with norm $\|\cdot\|_{(\varphi),p}$  defined as follows:
\begin{equation}
\|T\|_{(\varphi),p}=
\begin{cases} \left[\sum\limits_{j=1}^{2}\|T_{j}\|_{(\varphi_{j})}^{p}\right]^{\frac{1}{p}}, &1\leq p<\infty, \nonumber\\
\max\limits_{j}\|T_{j}\|_{(\varphi_{j})}, &p=\infty,
\end{cases}
\end{equation}
or the norm $\|\cdot\|_{\psi,p}$ defined in the same way as before which $\|\cdot\|_{\left(\varphi_{j}\right)}$ is replaced by the Orlicz norm $\|\cdot\|_{\psi_{j}}$. If $\psi_{j}$ is the complementary N-function of $\varphi_{j}$, denoted by $\bigoplus\limits_{j=1}^{2}S_{\psi_{j},q}$ which is equipped with $\|\cdot\|_{\psi,q}$ for $q=\frac{p}{p-1}$.
\end{definition}
\begin{remark}
If $T_{2}=0$, then $\bigoplus\limits_{j=1}^{2}S_{\varphi_{j},p}=S_{\varphi_{1}}.$
\end{remark}
\begin{theorem}
Assume $T\in \bigoplus\limits_{j=1}^{2}S_{\varphi_{j},p}$ and $B\in \bigoplus\limits_{j=1}^{2}S_{\psi_{j},q}$, where $1\leq p<\infty$ and $\frac{1}{p}+\frac{1}{q}=1$, one has

$(1)$ If $\|T_{j}\|_{(\varphi_{j})}\leq1$, then
$\upsilon(\varphi(T))={\rm tr}\left(
                                                                      \begin{array}{cc}
                                                                        \varphi_{1}(T_{1}) & 0 \\
                                                                        0 & \varphi_{2}(T_{2}) \\
                                                                      \end{array}
                                                                    \right)
\leq 2^{\frac{1}{q}}\cdot \|T\|_{(\varphi),p}$.

$(2)$ If $\|T_{j}\|_{(\varphi_{j})}>1$, then  $\upsilon(\varphi(T))={\rm tr}\left(
                                                                      \begin{array}{cc}
                                                                        \varphi_{1}(T_{1}) & 0 \\
                                                                        0 & \varphi_{2}(T_{2}) \\
                                                                      \end{array}
                                                                    \right)
>\|T\|_{(\varphi),p}$.

$(3)$\rm{(H$\ddot{\rm{o}}$lder inequality)}
$\upsilon(|TB|)={\rm tr}\left(
                                                                      \begin{array}{cc}
                                                                        T_{1}B_{1} & 0 \\
                                                                        0 & T_{2}B_{2} \\
                                                                      \end{array}
                                                                   \right)\leq\|T\|_{\varphi,p}\cdot\|B\|_{(\psi),q}$
or $\upsilon(|TB|)\leq\|T\|_{(\varphi),p}\cdot\|B\|_{\psi,q}$.

$(4)$ If $\varphi\in \delta(0)$, then there exists $k>0$, such that
$$\upsilon(\varphi(|T+B|))
\leq\frac{k}{2}\upsilon((\varphi(|T|)+\varphi(|B|)).$$

$(5)$ If $\|T\|_{(\varphi),p}\neq0$, then $$\upsilon\left(\varphi\left(\frac{T}{\|T\|_{(\varphi),p}}\right)\right)\leq 2^{1-\frac{1}{p}}.$$
\end{theorem}
\begin{proof}
$(1)$ If $\|T_{j}\|_{(\varphi_{j})}\leq1$,  by Proposition 3.4 of \cite{Ghadir} and classical H$\rm\ddot{o}$lder inequality, then we have
\begin{eqnarray*}
\upsilon(\varphi(T))&=&\sum\limits^{2}_{j=1}{\rm tr}(\varphi_{j}(T_{j}))\\
&=&\sum\limits^{2}_{j=1}1\cdot {\rm tr} (\varphi_{j}(T_{j}))\\
&\leq& \left(\sum\limits^{2}_{j=1}1^{q}\right)^{\frac{1}{q}}\cdot \left[\sum\limits^{2}_{j=1}\left({\rm tr} (\varphi_{j}(T_{j})\right)^{p}\right]^{\frac{1}{p}}\\
&\leq&2^{\frac{1}{q}}\cdot\left[\sum\limits^{2}_{j=1}\|T_{j}\|_{(\varphi_{j})}^{p}\right]^{\frac{1}{p}}\\
&=&2^{\frac{1}{q}}\cdot\|T\|_{(\varphi),p}.
\end{eqnarray*}

$(2)$ If $\|T_{j}\|_{(\varphi_{j})}>1$, by Proposition 3.4 of \cite{Ghadir}, we have
\begin{eqnarray*}
\left[\upsilon(\varphi(T))\right]^{p}&=&\left[\sum\limits^{2}_{j=1}{\rm tr}(\varphi_{j}(T_{j}))\right]^{p}\\
&>&\left[\sum\limits^{2}_{j=1}\|T_{j}\|_{(\varphi_{j})}\right]^{p}\\
&\geq&\sum\limits^{2}_{j=1}\|T_{j}\|^{p}_{(\varphi_{j})},
\end{eqnarray*}
which means that $$\upsilon(\varphi(T))>\left[\sum\limits^{2}_{j=1}\|T_{j}\|^{p}_{(\varphi_{j})}\right]^{\frac{1}{p}}=\|T\|_{(\varphi),p}.$$

$(3)$ By Theorem 3.3 of \cite{Ghadir} and  Theorem 8 of \cite{Rao}, one get that
\begin{eqnarray*}
\upsilon(|TB|)&=&\sum\limits^{2}_{j=1}\left|{\rm tr}(T_{j}B_{j})\right|\\
&\leq&\sum\limits^{2}_{j=1}\left(\|T_{j}\|_{\varphi_{j}}\cdot\|B_{j}\|_{(\psi_{j})}\right)\\
&\leq&\left(\sum\limits^{2}_{j=1}\|T_{j}\|^{p}_{\varphi_{j}}\right)^{\frac{1}{p}}
\cdot\left(\sum\limits^{2}_{j=1}\|B_{j}\|^{q}_{(\psi_{j})}\right)^{\frac{1}{q}}\\
&=&\|T\|_{\varphi,p}\cdot \|B\|_{(\psi),q}.
\end{eqnarray*}
The other inequality can be obtained similarly.

$(4)$ We know there exist partially isometric operators $U,V$ such that $|T+B|\leq U|T|U^{*}+V|B|V^{*}$, using the convexity of $\varphi$ and  $\varphi\in \delta(0)$, we can derive
\begin{eqnarray*}
\upsilon(\varphi(|T+B|))&=&\sum\limits^{2}_{j=1}{\rm tr}\left[\varphi_{j}(T_{j}+B_{j})\right]\\
&\leq& \sum\limits^{2}_{j=1}{\rm tr}\left(\varphi_{j}\left(U_{j}|T_{j}|U_{j}^{*}+V_{j}|B_{j}|V_{j}^{*}\right)\right)\\
&=&\sum\limits^{2}_{j=1}{\rm tr}\left(\varphi_{j}\left(2\frac{U_{j}|T_{j}|U_{j}^{*}+V_{j}|B_{j}|V_{j}^{*}}{2}\right)\right)\\
&\leq&\sum\limits^{2}_{j=1}\frac{k_{j}}{2}\left[{\rm tr}\left(\varphi_{j}\left(U_{j}|T_{j}|U_{j}^{*}+V_{j}|B_{j}|V_{j}^{*}\right)\right)\right]\\
&=&\sum\limits^{2}_{j=1}\left\{\frac{k_{j}}{2}\left[\sum_{n\geq1}\varphi_{j}\left(s_{n}(U_{j}|T_{j}|U_{j}^{*})\right)+\sum_{n\geq1}\varphi_{j}\left(s_{n}(V_{j}|B_{j}|V_{j}^{*})\right)\right]\right\}\\
&\leq&\sum\limits^{2}_{j=1}\left\{\frac{k_{j}}{2}\left[\sum_{n\geq1}\varphi_{j}\left(s_{n}(|T_{j}|)\right)+\sum_{n\geq1}\varphi_{j}\left(s_{n}(|B_{j}|)\right)\right]\right\}\\
&\leq&\frac{k}{2}\sum\limits^{2}_{j=1}{\rm tr}\left[\varphi_{j}(T_{j})+\varphi_{j}(B_{j})\right]\\
&=&\frac{k}{2}\left[\upsilon(\varphi(T))+\upsilon(\varphi(B))\right],
\end{eqnarray*}
where $k=\max\{k_{1},k_{2}\}$.

$(5)$ If $\|T\|_{(\varphi),p}\neq0$, then by (1) one has,
\begin{eqnarray*}
\upsilon\left(\varphi\left(\frac{T}{\|T\|_{(\varphi),p}}\right)\right)&=&{\rm tr}\left(
                                                                      \begin{array}{cc}
                                                                        \varphi_{1}\left(\frac{T_{1}}{\|T\|_{(\varphi),p}}\right) & 0 \\
                                                                        0 & \varphi_{2}\left(\frac{T_{2}}{\|T\|_{(\varphi),p}}\right) \\
                                                                      \end{array}
                                                                    \right)\\
&=&{\rm tr}\left(\varphi_{1}\left(\frac{T_{1}}{\|T\|_{(\varphi),p}}\right) \right)+{\rm tr}\left( \varphi_{2}\left(\frac{T_{2}}{\|T\|_{(\varphi),p}}\right) \right)\\
&\leq&\left\|\frac{T_{1}}{\|T\|_{(\varphi),p}}\right\|_{(\varphi_{1})}+\left\|\frac{T_{2}}{\|T\|_{(\varphi),p}}\right\|_{(\varphi_{2})}\\
&=&\frac{\|T_{1}\|_{(\varphi_{1})}}{\left(\|T\|^{p}_{(\varphi_{1})}+\|T\|^{p}_{(\varphi_{2})}\right)^{\frac{1}{p}}}+\frac{\|T_{2}\|_{(\varphi_{2})}}{\left(\|T\|^{p}_{(\varphi_{1})}+\|T\|^{p}_{(\varphi_{2})}\right)^{\frac{1}{p}}}\\
&\leq&2^{\frac{1}{q}}.
\end{eqnarray*}
In fact, we only need to consider the maximum value of the binary function $f(x,y)=\frac{(x+y)^{p}}{x^{p}+y^{p}}$, where $x,y\geq0$ and $p\geq1$.
\end{proof}
\begin{remark}
If $\varphi$ is 1-tuple of N-function , then Lemma 2.1 is exactly the Theorem 3.3 and Proposition 3.4 of \cite{Ghadir}.
\end{remark}
\begin{corollary}
$(1)$ $\upsilon\left(\frac{T^{p}}{\|T\|^{p}_{p,p}}\right)=1$, where $p\geq1.$
\begin{proof}
In (5) of theorem 2.1, if we let $\varphi_{1}(x)=\varphi_{2}(x)=|x|^{p}$, then
\begin{eqnarray*}
\upsilon\left(\frac{T^{p}}{\|T\|^{p}_{p,p}}\right)&=&{\rm tr}\left(
                                                                      \begin{array}{cc}
                                                                        \frac{\|T_{1}\|^{p}}{\|T_{1}\|^{p}+\|T_{2}\|^{p}} & 0 \\
                                                                        0 & \frac{\|T_{2}\|^{p}}{\|T_{1}\|^{p}+\|T_{2}\|^{p}} \\
                                                                      \end{array}
                                                                    \right)\\
&=&\frac{1}{\|T_{1}\|^{p}+\|T_{2}\|^{p}}\sum_{n\geq1}\left[s_{n}^{p}(T_{1})+s_{n}^{p}(T_{2})\right]\\
&=&\frac{\|T_{1}\|^{p}+\|T_{2}\|^{p}}{\|T_{1}\|^{p}+\|T_{2}\|^{p}}\\
&=&1.
\end{eqnarray*}
\end{proof}
$(2)$ If $T,B\in \bigoplus\limits_{j=1}^{2}S_{p,p}$ and $p\geq1$, then
$$\upsilon\left(|T+B|^{p}\right)\leq 2^{p-1}\sum_{n\geq1}\left[s_{n}^{p}(T)+s_{n}^{p}(B)\right].$$
\begin{proof}
In (4) of theorem 2.1, if we let $\varphi_{1}(x)=\varphi_{2}(x)=|x|^{p}$, then $k_{1}=k_{2}=2^{p}$ and
\begin{eqnarray*}
\upsilon\left(|T+B|^{p}\right)&=&{\rm tr}\left[(T_{1}+B_{1})^{p}+(T_{2}+B_{2})^{p}\right]\\
&=&\sum_{n\geq1}\left[s_{n}^{p}(T_{1}+B_{1})+s_{n}^{p}(T_{2}+B_{2})\right]\\
&\leq&\frac{2^{p}}{2}\left[\sum_{n\geq1}\left[s_{n}^{p}(T_{1})+s_{n}^{p}(T_{2})\right]+\sum_{n\geq1}\left[s_{n}^{p}(B_{1})+s_{n}^{p}(B_{2})\right]\right]\\
&=&2^{p-1}\sum_{n\geq1}\left[s_{n}^{p}(T)+s_{n}^{p}(B)\right],
\end{eqnarray*}
which implies that
$$\|T+B\|_{p}^{p}\leq2^{p-1}\left(\|T\|_{p}^{p}+\|B\|_{p}^{p}\right).$$
\end{proof}
$(3)${\rm (H$\ddot{\rm{o}}$lder inequality of $\bigoplus\limits_{j=1}^{2}S_{p,p}$)}  If $T\in \bigoplus\limits_{j=1}^{2}S_{p,p}, B\in \bigoplus\limits_{j=1}^{2}S_{q,q}$, where $\frac{1}{p}+\frac{1}{q}=1$ with $p\geq1$, then we have
$$\upsilon(|TB|)\leq \|T\|_{p}\cdot \|B\|_{q}.$$
\begin{proof}
In (3) of theorem 2.1,  $\varphi_{1}(x)=\varphi_{2}(x)=|x|^{p}$, we can infer
\begin{eqnarray*}
\upsilon(|TB|)&=&\sum^{2}_{n=1}{\rm tr}(T_{j}B_{j})\\
&\leq&\sum^{2}_{n=1}\|T_{j}\|_{p}\|B_{j}\|_{q}\\
&\leq&\left(\sum^{2}_{n=1}\|T_{j}\|^{p}_{p}\right)^{\frac{1}{p}}\cdot\left(\sum^{2}_{n=1}\|B_{j}\|^{q}_{q}\right)^{\frac{1}{q}}\\
&=&\|T\|_{p}\cdot \|B\|_{q}.
\end{eqnarray*}
\end{proof}
\end{corollary}
\begin{theorem}
$\bigoplus\limits_{j=1}^{2}S_{\varphi_{j},p}$ is a two sides ideal of $B(\mathcal{H})\bigoplus B(\mathcal{H})$, which means that for any $T\in \bigoplus\limits_{j=1}^{2}S_{\varphi_{j},p}$ and $B=(B_{1},B_{2}),C=(C_{1},C_{2})\in B(\mathcal{H})\bigoplus B(\mathcal{H})$, one has $\widetilde{B}T\widetilde{C}, \widetilde{C}T\widetilde{B}\in \bigoplus\limits_{j=1}^{2}S_{\varphi_{j},p}$ and $$\|\widetilde{B}T\widetilde{C}\|_{(\varphi),p}\leq \|\widetilde{B}\|_{\infty}\|T\|_{(\varphi),p}\|\widetilde{C}\|_{\infty},$$
where $\widetilde{B}=(\widetilde{B_{1}},\widetilde{B_{1}})$ with $\widetilde{B_{1}}=\{B_{j}: \max\{\|B_{j}\|_{\infty}\}, j=1,2\}$ and $\widetilde{C}=(\widetilde{C_{1}},\widetilde{C_{1}})$ with $\widetilde{C_{1}}=\{C_{j}: \max\{\|C_{j}\|_{\infty}\}, j=1,2\}$.
\end{theorem}
\begin{proof}
From (2) of Theorem 3.3 in \cite{Ma} we know, for any $T\in S_{\varphi}$ and $B\in B(\mathcal{H})$, we have $\|TB\|_{(\varphi)}\leq\|T\|_{(\varphi)}\|B\|_{\infty}$.

Hence, for any $T\in \bigoplus\limits_{j=1}^{2}S_{\varphi_{j},p}$ and $B=(B_{1},B_{2})\in B(\mathcal{H})\bigoplus B(\mathcal{H})$, one has
\begin{eqnarray*}
\|TB\|_{(\varphi),p}&=&\left[\|T_{1}B_{1}\|^{p}_{(\varphi_{1})}+\|T_{2}B_{2}\|^{p}_{(\varphi_{2})}\right]^{\frac{1}{p}}\\
&\leq&\left[\|T_{1}\|^{p}_{(\varphi_{1})}\|B_{1}\|^{p}_{\infty}+\|T_{2}\|^{p}_{(\varphi_{2})}\|B_{2}\|^{p}_{\infty}\right]^{\frac{1}{p}}\\
&\leq&\left[\|T_{1}\|^{p}_{(\varphi_{1})}+\|T_{2}\|^{p}_{(\varphi_{2})}\right]^{\frac{1}{p}}\|\widetilde{B}\|_{\infty}\\
&=&\|T\|_{(\varphi),p}\cdot \|\widetilde{B}\|_{\infty}.
\end{eqnarray*}
 By the same method,  for any $C=(C_{1},C_{2})\in B(\mathcal{H})\bigoplus B(\mathcal{H})$, since $BT\in \bigoplus\limits_{j=1}^{2}S_{\varphi_{j},p}$, one get that $BTC\in \bigoplus\limits_{j=1}^{2}S_{\varphi_{j},p}$ with $$\|\widetilde{B}T\widetilde{C}\|_{(\varphi),p}\leq \|\widetilde{B}\|_{\infty}\|T\|_{(\varphi),p}\|\widetilde{C}\|_{\infty}.$$
\end{proof}
\begin{theorem}
If $T\in \bigoplus\limits_{j=1}^{2}S_{\varphi_{j},p}$, then for $1\leq p<\infty$, the  norm $\|\cdot\|_{(\varphi),p}$ is given by
$$\|T\|_{(\varphi),p}=\sup\left\{\upsilon(|TB|): B\in \bigoplus\limits_{j=1}^{2}S_{\psi_{j},q},\|B\|_{\psi,q}\leq1\right\}.$$
\end{theorem}
\begin{proof}
If $\|B\|_{\psi,q}\leq1$. One side, by (3) of Theorem 2.1, we have
$$\upsilon(|TB|)\leq\|T\|_{(\varphi),p}\cdot \|B\|_{\psi,q}\leq \|T\|_{(\varphi),p}.$$
On the other side, we may take, for simplicity, that $T_{j}\geq0$ and  $\|T_{j}\|_{(\varphi_{j})}=1$, then $\|T\|_{(\varphi),p}=2^{\frac{1}{p}}$.

By Proposition 3.4 of \cite{Ghadir}, for any $1>\varepsilon>0$ there exists $\delta\in(0,\varepsilon)$ such that $$1-\delta\leq{\rm tr}\left[\varphi_{j}\left((1-\varepsilon)T_{j}\right)\right]\leq\|(1-\varepsilon)T_{j}\|_{(\varphi_{j})}=1-\varepsilon.$$

Recalling that $p$ is the left derivative of $\varphi$, if we set $$B_{j}=\frac{p_{j}((1-\varepsilon)T_{j})}{2^{\frac{1}{q}}[1+{\rm tr}(\psi_{j}( p_{j}((1-\varepsilon)T_{j})))]},$$
then  $B=(B_{1},B_{2})\in \bigoplus\limits_{j=1}^{2}S_{\psi_{j},q}$. Moreover by 1.7 of \cite{Zhenhua} and 1.9 of \cite{Chen}(Young Inequality) we have
\begin{eqnarray*}
{\rm tr}\left(T_{j}\cdot B_{j}\right)&=&\frac{{\rm tr}\left((1-\varepsilon)T_{j}\cdot p_{j}((1-\varepsilon)T_{j})\right)}{(1-\varepsilon)2^{\frac{1}{q}}[1+{\rm tr}(\psi_{j}( p_{j}((1-\varepsilon)T_{j})))]}\\
&=&\frac{{\rm tr}\left(\varphi_{j}((1-\varepsilon)T_{j})\right)+{\rm tr}\left(\psi_{j}\left(p_{j}((1-\varepsilon)T_{j})\right)\right)}{(1-\varepsilon)2^{\frac{1}{q}}[1+{\rm tr}(\psi_{j}( p_{j}((1-\varepsilon)T_{j})))]}\\
&\leq&\frac{1+{\rm tr}\left(\psi_{j}\left(p_{j}((1-\varepsilon)T_{j})\right)\right)}{(1-\varepsilon)2^{\frac{1}{q}}[1+{\rm tr}(\psi_{j}( p_{j}((1-\varepsilon)T_{j})))]}\\
&=&\frac{1}{2^{\frac{1}{q}}(1-\varepsilon)}.
\end{eqnarray*}
Thus, $\|B_{j}\|_{\psi_{j}}\leq 2^{\frac{1}{q}}$ since $\varepsilon$ is arbitrary.

Hence,
$$\|B\|_{\psi,q}=\left(\sum\limits^{2}_{j=1}\|B_{j}\|_{\psi_{j}}^{q}\right)^{\frac{1}{q}}\leq1.$$

However, one has
\begin{eqnarray*}
\sup\{\upsilon(|TB|),\|B\|_{\psi,q}\leq1\}&=&\sup\left\{\sum\limits^{n}_{j=1}{\rm tr}(T_{j}B_{j}):B_{j}\in \bigoplus\limits_{j=1}^{2}S_{\psi_{j},q},\|B\|_{\psi,q}\leq1\right\}\\
&>&\sup\left\{\sum\limits^{2}_{j=1}{\rm tr}((1-\varepsilon)T_{j} B_{j}):B_{j}\in \bigoplus\limits_{j=1}^{2}S_{\psi_{j},q}, \|B\|_{\psi_{j},q}\leq1\right\}\\
&=&\frac{1}{2^{\frac{1}{q}}}\sup\left\{\sum\limits^{2}_{j=1}\frac{{\rm tr}(\varphi_{j}(1-\varepsilon)T_{j})+{\rm tr}(\psi_{j}(p(1-\varepsilon)T_{j}))}{1+{\rm tr}(\psi_{j}(p(1-\varepsilon)T_{j}))}\right\}\\
&\geq&\frac{1}{2^{\frac{1}{q}}}\sup\left\{\sum\limits^{2}_{j=1}\frac{(1-\delta)+{\rm tr}(\psi_{j}(p(1-\varepsilon)T_{j}))}{1+{\rm tr}(\psi_{j}(p(1-\varepsilon)T_{j}))}\right\}\\
&>&\frac{2}{2^{\frac{1}{q}}}\\
&=&2^{\frac{1}{p}}\\
&=&\|T\|_{(\varphi),p}.
\end{eqnarray*}
from which the conlusion can be desired.
\end{proof}
\section{Riesz-Thorin Interpolation Theorem }
\begin{definition}\cite{Cleaver}
Let $\varphi_{1}$ and $\varphi_{2}$ be a pair of N-functions,  and $0\leq s\leq1$ be fixed. Then $\varphi_{s}$ is the uniquely defined  inverse of $\varphi_{s}^{-1}(u)=[\varphi_{1}^{-1}(u)]^{1-s}[\varphi_{2}^{-1}(u)]^{s}$ for $u\geq0$, where $\varphi_{i}^{-1}$ is the uniquely inverse of the N-function $\varphi_{i}$, and $\varphi_{s}$ is called an intermediate function.
\end{definition}
\begin{lemma}(Three-Lines-theorem)\cite{Rao}
Let $S=\{z=x+iy: 0<x<1\}$ be the strip in $\mathbb{C}$. If $f$ is an analytic function on $S$ which is bounded and continuous on the closure $\overline{S}$ of $S$, and let the bound of $f$ on the vertical line at $x=s$ in $S$ be $M_{s}$, i.e., $M_{s}=\sup\{|f(s+iy)|: y\in \mathds{R}\}$, then one has $$M_{s}\leq M_{0}^{1-s}M_{1}^{s},\\ 0\leq s\leq1.$$
\end{lemma}
\begin{theorem}
Let $\varphi_{i}=(\varphi_{i1},\varphi_{i2}),\phi_{i}=(\phi_{i1},\phi_{i2}),i=1,2$ be 2-tuples of N-functions and $0\leq r_{1},r_{2},t_{1},t_{2}\leq\infty$. Next let $\varphi_{s}=(\varphi_{s1},\varphi_{s2}),\phi_{s}=(\phi_{s1},\phi_{s2})$ be the associated intermediate N-functions, $$\frac{1}{r_{s}}=\frac{1-s}{r_{1}}+\frac{s}{r_{2}}, \frac{1}{t_{s}}=\frac{1-s}{t_{1}}+\frac{s}{t_{2}}, 0\leq s\leq1.$$
If $F:\bigoplus\limits_{j=1}^{2}S_{\varphi_{ij},r_{i}}\rightarrow \bigoplus\limits_{j=1}^{2}S_{\phi_{ij},t_{i}}$ is a bounded linear operator with bounds $K_{1},K_{2}$ satisfying
$$\|FT\|_{(\phi_{i}),t_{i}}\leq K_{i}\|T\|_{(\varphi_{i}),r_{i}}, T\in \bigoplus\limits_{j=1}^{2}S_{\varphi_{ij},r_{i}}, i=1,2,$$
then $F$ is also defined on $\bigoplus\limits_{j=1}^{2}S_{\varphi_{sj},r_{s}}$ into $\bigoplus\limits_{j=1}^{2}S_{\phi_{sj},t_{s}}$ for all $0\leq s\leq1$ and one has,
$$\|FT\|_{(\phi_{s}),t_{s}}\leq K_{1}^{1-s}K_{2}^{s}\|T\|_{(\varphi_{s}),r_{s}}.$$
\end{theorem}
\begin{proof}
Like in the commutative case, the proof is based on the Lemma 3.1. Let $T=(T_{1},T_{2})\in \bigoplus\limits_{j=1}^{2}S_{\varphi_{ij},r_{i}}, B=(B_{1},B_{2})\in\bigoplus\limits_{j=1}^{2}S_{\psi_{ij},t_{i}}$. By approximation, both $|T|$ and $|B|$ are linear combinations of mutually orthogonal projections, hence, $$T_{k}=U_{k}|T_{k}|,B_{k}=V_{k}|B_{k}|,$$ where $|T_{k}|=\sum\limits_{j=1}^{n}\alpha_{j}e_{kj}, |B_{k}|=\sum\limits_{j=1}^{n}\beta_{j}e'_{kj}, k=1,2$ and $\alpha_{j},\beta_{j}\in (0,\infty) $.
For convenience, we assume that $\|T\|_{(\varphi_{i}),r_{i}}\leq1, \|B\|_{\psi_{i},t_{i}}\leq1$.

For $z=\mathds{C}$ and $k=1,2$, we define
$$T(z)=(T_{1}(z),T_{2}(z))$$  and $$B(z)=(B_{1}(z),B_{2}(z)),$$
where
$$T_{k}(z)=U_{k}\varphi_{sk}\left[(\varphi_{1k}^{-1})^{1-z}(\varphi_{2k}^{-1})^{z}\right](|T_{k}|),$$
$$B_{k}(z)=V_{k}\psi_{sk}\left[(\psi_{1k}^{-1})^{1-z}(\psi_{2k}^{-1})^{z}\right](|B_{k}|).$$
Then,
\begin{eqnarray*}
T_{k}(z)&=&U_{k}\varphi_{sk}\left[\left(\varphi_{1k}^{-1}\left(\sum_{j=1}^{n}\alpha_{j}e_{kj}\right)\right)^{1-z}\left(\varphi_{2k}^{-1}\left(\sum\limits_{j=1}^{n}\alpha_{j}e_{kj}\right)\right)^{z}\right]\\
&=&\sum\limits_{j=1}^{n}\varphi_{sk}\left[\left(\varphi_{1k}^{-1}(\alpha_{j})\right)^{1-z}\left(\varphi_{2k}^{-1}(\alpha_{j})\right)^{z}\right]U_{k} e_{kj}.
\end{eqnarray*}
Hence, $z\mapsto T(z)$ is an analytic function on $\mathds{C}$ with value in $B{(\mathcal{H})}$. The same reduction applies to $B$.

Now we define a bounded entire function $$H(z)=\upsilon(B(z)FT(z)).$$
If $z=it$ for $t\in \mathds{R}$, we have
\begin{eqnarray*}
T_{k}(it)
&=&\sum\limits_{j=1}^{n}\varphi_{sk}[\left(\varphi_{1k}^{-1}(\alpha_{j})\right)^{1-it}(\varphi_{2k}^{-1}(\alpha_{j}))^{it}]U_{k} e_{kj}\\
&=&\sum\limits_{j=1}^{n}\varphi_{sk}\left[\left(\frac{\varphi_{2k}^{-1}(\alpha_{j})}{\varphi_{1k}^{-1}(\alpha_{j})}\right)^{it}\right]U_{k} e_{kj}\cdot \sum\limits_{j=1}^{n}\varphi_{sk}\left[\varphi_{1k}^{-1}(\alpha_{j})\right]U_{k} e_{kj}\\
&=&\left[\varphi_{sk}\left(\frac{\varphi_{2k}^{-1}}{\varphi_{1k}^{-1}}(|T_{k}|)\right)\right]^{it}\cdot \varphi_{sk}\left(\varphi_{1k}^{-1}(|T_{k}|)\right).
\end{eqnarray*}
Hence,$$|T_{k}(it)|^{2}=T_{k}(it)^{*}T_{k}(it)=\left[\varphi_{sk}\left(\varphi_{1k}^{-1}(|T_{k}|)\right)\right]^{2}$$ which means $$|T_{k}(it)|=\varphi_{sk}\left(\varphi_{1k}^{-1}(|T_{k}|)\right).$$
Hence, we have $\varphi_{1k}(|T_{k}(it)|)=\varphi_{sk}(|T_{k}|)$ which implies that $$\|T_{k}(it)\|_{(\varphi_{1k})}=\|T_{k}\|_{(\varphi_{sk})}$$ and
\begin{eqnarray*} \upsilon(\varphi_{1}(T(it)))&=&\sum\limits_{j=1}^{2}{\rm tr}\left[\varphi_{1j}\left[\varphi_{sj}\left(\varphi_{1j}^{-1}(|T_{j}|)\right)\right]\right]\\
&=&{\rm tr}(\varphi_{s1}(|T_{1}|))+{\rm tr}(\varphi_{s2}(|T_{2}|))\\
&=&\upsilon(\varphi_{s}(|T|)),
\end{eqnarray*}
where $$\|T(it)\|_{(\varphi_{1}),r_{s}}=\|T\|_{(\varphi_{s}),r_{s}}\leq K_{1}.$$
Similarly $\|B(it)\|_{\psi_{1},t_{s}}=\|B\|_{\psi_{s},t_{s}}\leq1$.
Thus by (3) of the Lemma 2.1 and the assumption on $F$, one has
$$\upsilon(|B(it)FT(it)|)\leq  K_{1}\|B(it)\|_{\psi_{1},t_{s}}\|T(it)\|_{(\varphi_{1}),r_{s}}\leq K_{1}.$$
It then follows that $|H(it)|\leq 1$ for any $t\in \mathds{R}$.
 In the same way, we show $|H(1+it)|\leq K_{2}$.

 Therefore, for any $0\leq s\leq1$, from Lemma 3.1 one get $$|H(s+it)|=\upsilon(|B(s+it)FT(s+it)|)\leq K_{1}^{1-s}K_{2}^{s}.$$
Combining with Theorem 2.3 one has $$\|FT\|_{(\phi_{s}),t_{s}}\leq K_{1}^{1-s}K_{2}^{s}\|T\|_{(\varphi_{s}),r_{s}},$$
as desired.
\end{proof}
The following result is noncommutative Orlicz sequence spaces analogs of the Clarkson inequality for the $S_{p}$.
\begin{theorem}
If $\varphi_{s}$ is an intermediate
function of $\varphi_{1},\varphi_{2}$ and $T=(T_{1},T_{2})\in \bigoplus\limits_{j=1}^{2}S_{\varphi_{ij},r_{i}}$, one has
$$\left(\|T_{1}+T_{2}\|_{(\varphi_{s})}^{\frac{2}{s}}+\|T_{1}-T_{2}\|_{(\varphi_{s})}^{\frac{2}{s}}\right)^{\frac{s}{2}}
\leq 2^{\frac{s}{2}}\left(\|T_{1}\|_{(\varphi_{s})}^{\frac{2}{2-s}}+\|T_{2}\|_{(\varphi_{s})}^{\frac{2}{2-s}}\right)^{\frac{2-s}{2}}.$$
\end{theorem}
\begin{proof}
To further identify with the proceeding theorem, we take $\phi_{1}=\varphi_{1}=(\varphi,\varphi),\phi_{2}=\varphi_{2}=(\varphi_{0},\varphi_{0})$ and for $0< s\leq1$, $$\varphi_{s}^{-1}(u)=\left[\varphi^{-1}(u)\right]^{1-s}\left[\varphi_{0}^{-1}(u)\right]^{s}
=\left[\varphi^{-1}(u)\right]^{1-s}u^{\frac{s}{2}},$$ where $\varphi_{0}(u)=u^{2}$.

Hence,
\begin{equation}
\|T\|_{(\varphi_{1}),r_{1}}=
\begin{cases} \left[\|T_{1}\|_{(\varphi)}^{r_{1}}+\|T_{2}\|_{(\varphi)}^{r_{1}}\right]^{\frac{1}{r_{1}}}, &1\leq r_{1}<\infty, \nonumber\\
\max\{\|T_{1}\|_{(\varphi)},\|T_{2}\|_{(\varphi)}\}, &r_{1}=\infty.
\end{cases}
\end{equation}
Set $r_{1}=1,r_{2}=t_{2}=2$ and $t_{1}=+\infty$ and define the linear operator $F: \bigoplus\limits_{j=1}^{2}S_{\varphi_{i},r_{i}}\rightarrow \bigoplus\limits_{j=1}^{2}S_{\phi_{i},t_{i}}$ by the equation $F(T_{1},T_{2})=(T_{1}+T_{2},T_{1}-T_{2})$. Then it follows that
\begin{eqnarray*} \|FT\|_{(\phi_{1}),t_{1}}&=&\max\{\|T_{1}+T_{2}\|_{(\varphi)},\|T_{1}-T_{2}\|_{(\varphi)}\}\\
&\leq& \|T_{1}\|_{(\varphi)}+\|T_{2}\|_{(\varphi)}\\
&=&K_{1}\|T\|_{(\varphi_{1}),r_{1}}.
\end{eqnarray*}
Hence, $K_{1}=1$  and since $\|\cdot\|_{(\varphi_{0})}=\|\cdot\|_{2}$, we find
\begin{eqnarray*} \|FT\|_{(\phi_{2}),t_{2}}&=&\left[\|T_{1}+T_{2}\|^{2}_{2}+\|T_{1}-T_{2}\|^{2}_{2}\right]^{\frac{1}{2}}\\
&=& \sqrt{2}\left[\|T_{1}\|^{2}_{2}+\|T_{2}\|^{2}_{2}\right]^{\frac{1}{2}}\\
&=&K_{2}\|T\|_{(\varphi_{2}),r_{2}},
\end{eqnarray*}
thus $K_{2}=\sqrt{2}$.

Let $r_{s}$ and $t_{s}$ be given by $$\frac{1}{r_{s}}=\frac{1-s}{r_{1}}+\frac{s}{r_{2}}, \frac{1}{t_{s}}=\frac{1-s}{t_{1}}+\frac{s}{t_{2}}$$
then we have, $r_{s}=\frac{2}{2-s}, t_{s}=\frac{2}{s}$.

By the  Theorem 3.1 it follows that
$$\|FT\|_{(\phi_{s}),t_{s}}\leq 2^{\frac{s}{s}}\|T\|_{(\varphi_{s}),r_{s}},$$
since $K_{1}^{1-s}K_{2}^{s}=2^{\frac{s}{2}}$.

Hence, we have
$$\|T\|_{(\phi_{s}),r_{s}}=\left[\|T_{1}\|_{(\varphi_{s})}^{\frac{2}{2-s}}+\|T_{2}\|_{(\varphi_{s})}^{\frac{2}{2-s}}\right]^{\frac{2-s}{2}}$$
and
$$\|FT\|_{(\phi_{s}),t_{s}}=\left(\|T_{1}+T_{2}\|_{(\varphi_{s})}^{\frac{2}{s}}+\|T_{1}-T_{2}\|_{(\varphi_{s})}^{\frac{2}{s}}\right)^{\frac{s}{2}}.$$
\end{proof}
The following corollary is Clarkson inequality of noncommutative $S_{p}$ space and the proof is similar with the P42 of \cite{Rao}.
\begin{corollary}
Suppose that $1<p<\infty$ and $q=\frac{p}{p-1}$. Then for $T_{1},T_{2}\in S_{p}$, we have
$$\left(\|T_{1}+T_{2}\|_{p}^{q}+\|T_{1}-T_{2}\|_{p}^{q}\right)^{\frac{1}{q}}\leq2^{\frac{1}{q}}\left(\|T_{1}\|_{p}^{p}+\|T_{2}\|_{p}^{p}\right)^{\frac{1}{p}}, 1<p\leq2,$$
and
$$\left(\|T_{1}+T_{2}\|_{p}^{p}+\|T_{1}-T_{2}\|_{p}^{p}\right)^{\frac{1}{p}}\leq2^{\frac{1}{p}}\left(\|T_{1}\|_{p}^{q}+\|T_{2}\|_{p}^{q}\right)^{\frac{1}{q}}, 2\leq p\leq \infty.$$
\end{corollary}
\begin{proof}
If $1<p\leq2$, then let $1<\alpha<p\leq2$ and $\varphi(u)=|u|^{\alpha}, \varphi_{0}(u)=|u|^{2}, s=\frac{2(p-\alpha)}{p(2-\alpha)}$. Thus $0<s\leq1$ and $\varphi_{s}^{-1}(u)=|u|^{\frac{1}{p}}$ or $\varphi_{s}(u)=|u|^{p}$. Hence $\|\cdot\|_{(\varphi_{s})}=\|\cdot\|_{(p)}$ and since $\lim\limits_{\alpha\downarrow1}\frac{2}{s}=\frac{p}{p-1}=q$; $\lim\limits_{\alpha\downarrow1}\frac{2-s}{2}=\frac{1}{p}$ by the Theorem 3.2 one obtains the first inequality.

Similarly let $2\leq p<\beta<\infty$  and $\varphi(u)=|u|^{\beta}, \varphi_{0}(u)=|u|^{2}, s=\frac{2(\beta-p)}{p(\beta-2)}$. Then $0\leq s\leq1$ and $\varphi_{s}(u)=|u|^{p},$ $\lim\limits_{\beta\uparrow\infty}\frac{2}{s}=p$; $\lim\limits_{\beta\uparrow\infty}\frac{2-s}{2}=\frac{1}{q}$. By the Theorem 3.2 again  one gets the second inequality.
\end{proof}
\section{The von Neumann-Jordan constant}
Based on the classical result of Jordan and von Neumann, we will calculate the von Neumann-Jordan constant of $S_{\varphi}$ and $S_{p}$.
\begin{definition}\cite{Clarkson}
The von Neumann-Jordan constant, $c_{NJ}(X)$ of a Banach space $X$ is the smallest $c>0$ such that for all $x,y\in X-\{0\}$,
$$\frac{1}{c}\leq F(x,y)=\frac{\|x+y\|^{2}+\|x-y\|^{2}}{2(\|x\|^{2}+\|y\|^{2})}\leq c.$$
\end{definition}
An equivalent definition of the von Neumann-Jordan constant is
$$c_{NJ}(X)=\sup\left\{\frac{\|x+y\|^{2}+\|x-y\|^{2}}{2(\|x\|^{2}+\|y\|^{2})}: \|x\|=1,\|y\|\leq1\right\}.$$
It is clear that $1\leq c_{NJ}(X)\leq2$, and by the Jordan-von Neumann Theorem, $c_{NJ}(X)=1$ iff $X$ is a Hilbert space.
\begin{definition}
For a Banach space $X$, the parameter $J(X)$ is termed a nonsquare constant, where $$J(X)=\sup\{\min(\|x+y\|,\|x-y\|):\|x\|=\|y\|=1\}.$$
\end{definition}
It is easy to know that $X$ is uniformly nonsquare iff $J(X)<2$. M.Kato and Y. Takahashi gave the following useful result.
\begin{lemma}\cite{Kato}
For any Banach space $X$, one has
$$J(X)^{2}\leq2c_{NJ}(X).$$
\end{lemma}
Next, we will obtain lower and upper bounds for the nonsquare and von Neumann-Jordan constants for $S_{\varphi}$. Thus we must give the following indices for $\varphi$.
\begin{definition}
For Orlicz function $\varphi$, we define :
$$\alpha_{\varphi}=\lim\limits_{u\rightarrow0}\inf \frac{\varphi^{-1}(u)}{\varphi^{-1}(2u)},\beta_{\varphi}=\lim\limits_{u\rightarrow0}\sup \frac{\varphi^{-1}(u)}{\varphi^{-1}(2u)}.$$
\end{definition}
\begin{theorem}
If $\varphi\in\delta_{2}(0)$, then it follows that
$$J(S_{\varphi})\geq \max\{(\alpha_{\varphi})^{-1},2\beta_{\varphi}\}.$$
\end{theorem}
\begin{proof}
One side, by definition of $\alpha_{\varphi}$, for given $0<\varepsilon<1$, we can find a $T_{0}\in S_{\varphi}$ such that $0<{\rm tr}(T_{0})<\frac{\varepsilon}{2}$ and $\frac{\varphi^{-1}({\rm tr}(T_{0}))}{\varphi^{-1}({\rm tr}(2T_{0}))}<\alpha_{\varphi}+\varepsilon$, which means $${\rm tr}(T_{0})<\varphi((\alpha_{\varphi}+\varepsilon)\varphi^{-1}({\rm tr}(2T_{0}))).$$
Let $n_{0}>0$ be the smallest integer such that $n_{0}\leq \frac{1}{{\rm tr}(2T_{0})}<n_{0}+1$, and choose $C\in S_{\varphi}$ such that $2n_{0}{\rm tr}(T_{0})+{\rm tr}(\varphi(C))=1$.

Define
$$T_{1}={\rm diag}(\underbrace {\varphi^{-1}(2T_{0}),\cdots,\varphi^{-1}(2T_{0})}_{n_{0}},C,0,\cdots)$$ and
$$T_{2}={\rm diag}( \underbrace{0,\cdots,0}_{n_{0}+1},\underbrace{\varphi^{-1}(2T_{0}),\cdots,\varphi^{-1}(2T_{0})}_{n_{0}},C,0,\cdots),$$ then $T_{1},T_{2}\in S_{\varphi}$ and ${\rm tr}(\varphi(T_{1}))={\rm tr}(\varphi(T_{2}))=2n_{0}{\rm tr}(T_{0})+{\rm tr}\varphi(C)=1$, hence $\|T_{1}\|_{(\varphi)}=\|T_{2}\|_{(\varphi)}=1$ and we get
\begin{eqnarray*}
{\rm tr}\left(\varphi\left(\frac{(\alpha_{\varphi}+\varepsilon)(T_{1}-T_{2})}{1-\varepsilon}\right)\right)&=&{\rm tr}\left(\varphi\left(\frac{(\alpha_{\varphi}+\varepsilon)(T_{1}+T_{2})}{1-\varepsilon}\right)\right)\\
&>&\frac{1}{1-\varepsilon}{\rm tr}\left(\varphi\left((\alpha_{\varphi}+\varepsilon)(T_{1}+T_{2})\right)\right)\\
&=&\frac{1}{1-\varepsilon}\varphi\left((\alpha_{\varphi}+\varepsilon)\cdot{\rm tr}(T_{1}+T_{2})\right)\\
&=&\frac{1}{1-\varepsilon}\varphi\left((\alpha_{\varphi}+\varepsilon)\cdot(2n_{0}{\rm tr}(\varphi^{-1}(2T_{0}))+2{\rm tr}(\varphi(C))\right)\\
&>&\frac{1}{1-\varepsilon}\varphi\left((\alpha_{\varphi}+\varepsilon)\cdot2n_{0}{\rm tr}(\varphi^{-1}(2T_{0}))\right)\\
&>&\frac{{\rm tr}(T_{0})}{1-\varepsilon}\\
&>&\frac{1-2{\rm tr}(T_{0})}{1-\varepsilon}\\
&>&1.
\end{eqnarray*}
Hence $$J(S_{\varphi})\geq \min\{\|T_{1}+T_{2}\|_{(\varphi)},\|T_{1}-T_{2}\|_{(\varphi)}\}>\frac{1-\varepsilon}{\alpha_{\varphi}+\varepsilon},$$
which implies that $J(S_{\varphi})\geq \alpha_{\varphi}^{-1}$ by arbitrariness of $\varepsilon$.

On the other side, by definition of $\beta_{\varphi}$, for a given $0<\varepsilon<1$, we can find a $B_{0}\in S_{\varphi}$ such that $0<{\rm tr}(B_{0})<\frac{\varepsilon}{2}$ and $\frac{\varphi^{-1}({\rm tr}(B_{0}))}{\varphi^{-1}({\rm tr}(2B_{0}))}>\beta_{\varphi}-\frac{\varepsilon}{2}$, which means $${\rm tr}(B_{0})<\varphi\left(\frac{2\varphi^{-1}({\rm tr}(B_{0}))}{2\beta_{\varphi}-\varepsilon)}\right).$$
Again let $n_{0}>0$ be the smallest integer such that $n_{0}\leq \frac{1}{{\rm tr}(2B_{0})}<n_{0}+1$, and choose $D\in S_{\varphi}$ such that $2n_{0}{\rm tr}(B_{0})+{\rm tr}(\varphi(D))=1$.

Define
$$T_{1}={\rm diag}(\underbrace {\varphi^{-1}(B_{0}),\cdots,\varphi^{-1}(B_{0})}_{2n_{0}},C,0,\cdots)$$ and
$$T_{2}={\rm diag}( \underbrace{\varphi^{-1}(B_{0}),\cdots,\varphi^{-1}(B_{0})}_{n_{0}},\underbrace{-\varphi^{-1}(B_{0}),\cdots,-\varphi^{-1}(B_{0})}_{n_{0}},0,D,0,\cdots),$$ then $T_{1},T_{2}\in S_{\varphi}$ and ${\rm tr}(\varphi(T_{1}))={\rm tr}(\varphi(T_{2}))=2n_{0}{\rm tr}(B_{0})+{\rm tr}\varphi(D)=1$, hence $\|T_{1}\|_{(\varphi)}=\|T_{2}\|_{(\varphi)}=1$ and we get
\begin{eqnarray*}
{\rm tr}\left(\varphi\left(\frac{T_{1}-T_{2}}{(1-\varepsilon)(2\beta_{\varphi}-\varepsilon)}\right)\right)&=&{\rm tr}\left(\varphi\left(\frac{T_{1}+T_{2}}{(1-\varepsilon)(2\beta_{\varphi}-\varepsilon)}\right)\right)\\
&>&\frac{2n_{0}{\rm tr}(B_{0})}{1-\varepsilon}\\
&>&1.
\end{eqnarray*}
Hence $$J(S_{\varphi})\geq \min\{\|T_{1}+T_{2}\|_{(\varphi)},\|T_{1}-T_{2}\|_{(\varphi)}\}>(1-\varepsilon)(2\beta_{\varphi}-\varepsilon),$$
which implies that $J(S_{\varphi})\geq 2\beta_{\varphi}$ by arbitrariness of $\varepsilon$.
\end{proof}
\begin{theorem}
Let $\varphi$ be an N-function and $\varphi_{s}$ be the inverse which satisfies that $\varphi_{s}^{-1}(u)=\left[\varphi^{-1}(u)\right]^{1-s}\left[\varphi_{0}^{-1}(u)\right]^{s}
=\left[\varphi^{-1}(u)\right]^{1-s}u^{\frac{s}{2}}$, where $0< s\leq1$ and $\varphi_{0}(u)=u^{2}$, then
$$c_{NJ}(S_{\varphi_{s}})\leq 2^{1-s}.$$
\end{theorem}
\begin{proof}
By Theorem 3.2, for any $T_{1},T_{2}\in S_{\varphi_{s}}$ we have $$\left(\|T_{1}+T_{2}\|_{(\varphi_{s})}^{\frac{2}{s}}+\|T_{1}-T_{2}\|_{(\varphi_{s})}^{\frac{2}{s}}\right)^{\frac{s}{2}}
\leq 2^{\frac{s}{2}}\left(\|T_{1}\|_{(\varphi_{s})}^{\frac{2}{2-s}}+\|T_{2}\|_{(\varphi_{s})}^{\frac{2}{2-s}}\right)^{\frac{2-s}{2}}.$$
Using the H${\rm \ddot{o}}$lder inequality for sequences and letting $p=\frac{1}{1-s},q=\frac{1}{s}$ one gets,
\begin{eqnarray*}
\|T_{1}+T_{2}\|_{(\varphi_{s})}^{2}+\|T_{1}-T_{2}\|_{(\varphi_{s})}^{2}&\leq&
2^{\frac{1}{p}}\left[\|T_{1}+T_{2}\|_{(\varphi_{s})}^{2q}+\|T_{1}-T_{2}\|_{(\varphi_{s})}^{2q}\right]^{\frac{1}{q}}\\ &=&2^{1-s}\left[\|T_{1}+T_{2}\|_{(\varphi_{s})}^{\frac{2}{s}}+\|T_{1}-T_{2}\|_{(\varphi_{s})}^{\frac{2}{s}}\right]^{s}\\
&\leq&2\left[\|T_{1}\|_{(\varphi_{s})}^{\frac{2}{2-s}}+\|T_{2}\|_{(\varphi_{s})}^{\frac{2}{2-s}}\right]^{2-s}.
\end{eqnarray*}
Similarly, let $p=\frac{2-s}{1-s},q=2-s$, where $0<s\leq1$, then let
\begin{eqnarray*}
\|T_{1}\|_{(\varphi_{s})}^{\frac{2}{2-s}}+\|T_{2}\|_{(\varphi_{s})}^{\frac{2}{2-s}}&\leq&
2^{\frac{1}{p}}\left[\|T_{1}\|_{(\varphi_{s})}^{\frac{2q}{2-s}}+\|T_{2}\|_{(\varphi_{s})}^{\frac{2q}{2-s}}\right]^{\frac{1}{q}}\\ &=&2^{\frac{2-s}{1-s}}\left[\|T_{1}\|_{(\varphi_{s})}^{2}+\|T_{2}\|_{(\varphi_{s})}^{2}\right]^{\frac{1}{2-s}}.
\end{eqnarray*}
Hence, for $T_{1}\neq T_{2}\neq0$, one can get the conclusion.
\end{proof}
Using Lemma 4.1 and Theorem 4.1 as well as  Theorem 4.2, we can deduce the bounds of the von Neumann-Jordan constant for $S_{(\varphi_{s})}$ as follows:
\begin{corollary}
Let $\varphi_{s}, 0<s\leq1$ be as the inverse of $\varphi_{s}^{-1}$ as Theorem 4.1 and $\varphi_{s}\in \delta_{2}(0)$. Then we obtain
$$\max\left\{\frac{1}{2}\alpha_{\varphi_{s}}^{-2},2\beta_{\varphi_{s}}^{2}\right\}\leq c_{NJ}(S_{\varphi_{s}})\leq 2^{1-s}.$$
\end{corollary}

\begin{example}
For $p\geq1$, one has
$$c_{NJ}(S_{p})= \max\left\{2^{\frac{2}{p}-1},2^{1-\frac{2}{p}}\right\}.$$
\end{example}
\begin{proof}
On one side, by Theorem 4.2, if $1<p\leq2$, then let $1<\alpha<p\leq2$ and $\varphi(u)=|u|^{\alpha}, \varphi_{0}(u)=|u|^{2}, s=\frac{2(p-\alpha)}{p(2-\alpha)}$. Thus $0<s\leq1$ and $\varphi_{s}(u)=|u|^{p}$. Let $\alpha\rightarrow1$, then $1-s\rightarrow\frac{2}{p}-1$, and $c_{NJ}(S_{p})\leq2^{\frac{2}{p}-1}$. If $2\leq p<\beta<\infty$, and $\varphi(u)=|u|^{\beta}, \varphi_{0}(u)=|u|^{2}, s=\frac{2(\beta-p)}{p(\beta-2)}$, then $0<s\leq1$ and $\varphi_{s}(u)=|u|^{p}$. Let $\beta\rightarrow\infty$, then $1-s\rightarrow1-\frac{2}{p}$, and $c_{NJ}(S_{p})\leq2^{1-\frac{2}{p}}$. Hence,
$$c_{NJ}(S_{p})\leq\max\left\{2^{\frac{2}{p}-1},2^{1-\frac{2}{p}}\right\}.$$
On the other side,  for $\varphi(x)=|x|^{p}$,  $$\alpha_{\varphi}=\beta_{\varphi}=\lim\limits_{u\rightarrow0}\inf \frac{\varphi^{-1}(u)}{\varphi^{-1}(2u)}=2^{-\frac{1}{p}},$$
then $\frac{1}{2}\alpha_{\varphi}^{-2}=2^{\frac{2}{p}-1}$ and $2\beta_{\varphi}^{2}=2^{1-\frac{2}{p}}$ which imply $$c_{NJ}(S_{p})\geq\max\left\{2^{\frac{2}{p}-1},2^{1-\frac{2}{p}}\right\}$$
as desired.
\end{proof}
\section*{Acknowledgement} We want to express our gratitude to the referee for all his/her careful revision and suggestions which has improved the final version of this work.
\section*{References}

\bibliography{mybibfile}

\end{spacing}
\end{document}